\documentclass[11pt]{amsart}%
\usepackage{enumerate}
\usepackage{amsfonts}
\usepackage{amsmath}
\usepackage{amssymb}
\usepackage{graphicx}%
\usepackage[final]{optional}
\usepackage[centertags]{amsmath}%
\usepackage{amsfonts,amssymb,amsthm, stmaryrd , graphicx, mathrsfs, mathpazo}%
\usepackage[usenames]{color}
\usepackage[toc,page]{appendix}

\newtheorem{thm}{Theorem}[section]
\newcommand{\bt}{\begin{thm}}
\newcommand{\et}{\end{thm}}

\newtheorem{cor}[thm]{Corollary} 
\newcommand{\bc}{\begin{cor}}
\newcommand{\ec}{\end{cor}}
\newtheorem{lem}[thm]{Lemma}  
\newcommand{\bl}{\begin{lem}}
\newcommand{\el}{\end{lem}}
\newtheorem{prop}[thm]{Proposition}
\newcommand{\bp}{\begin{prop}}
\newcommand{\ep}{\end{prop}}
\newtheorem{defn}[thm]{Definition}
\newcommand{\bd}{\begin{defn}}  
\newcommand{\ed}{\end{defn}}

\newtheorem{rmrk}[thm]{Remark}
\newcommand{\br}{\begin{rmrk}}
\newcommand{\er}{\end{rmrk}}

\newcommand{\be}{\begin{equation}}
\newcommand{\ee}{\end{equation}}

\newcommand{\R}{\mathbb{R}}

\newcommand{\diam}{\operatorname{diam}}

\begin{document}
\title[Fundamental Groups in $RCD(0,N)$]{ Non-branching $RCD(0,N)$ Geodesic Spaces with Small Linear Diameter Growth have Finitely Generated Fundamental Groups}

\author{Yu Kitabeppu}
\address[Yu Kitabeppu]{Kyoto University}
\email{y.kitabeppu@gmail.com}

\author{Sajjad Lakzian}
\address[Sajjad Lakzian]{HCM, Universit\"{a}t Bonn}
\email{lakzians@gmail.com}

\begin{abstract}
In this paper, we generalize the finite generation result of Sormani~\cite{Sormani-fund-group} to non-branching $RCD(0,N)$ geodesic spaces (and in particular, Alexandrov spaces) with full support measures. This is a special case of the Milnor's Conjecture for complete non-compact $RCD(0,N)$ spaces. One of the key tools we use is the Abresch-Gromoll type excess estimates for non-smooth spaces obtained by Gigli-Mosconii in~\cite{Gigli-Mosconi-Excess}.
\end{abstract}

\maketitle






\section{Introduction}

In~\cite{MR0232311}, Milnor conjectures that a complete non-compact Riemannian manifold, $M^n$, with non-negative Ricci curvature possesses a finitely generated fundamental group. The finite generation of fundamental group has been proven in the following cases:

\textbf{(i)} If $M$ has non-negative sectional curvature (Cheeger-Gromoll~\cite{ChGrNonnegCurvature});

\textbf{(ii)} When $M$ is three dimensional and $Ric > 0$ (Schoen-Yau~\cite{Schoen-Yau-Nonneg});

\textbf{(iii)} When $M$ has Euclidean volume growth (Anderson~\cite{MR1046624} and Li~\cite{MR847950});

\textbf{(iv)} $M^n$ has small diameter growth ($O(r^{\frac{1}{n}})$) and sectional curvature bounded below (Abresch-Gromoll~\cite{AbreschGromoll}).

Notice that when $n = 1,2$, the result follows from \textbf{(i)} since $Ric \ge 0$ is the same as non-negative sectional curvature.

As far as finite generation results in non-smooth spaces satisfying curvature-dimension bounds, Bacher-Sturm in~\cite{BacherSturm} prove the finite generation of the fundamental group for $CD(K,N)$ spaces with $K > 0$. This is a result of the Myer's Theorem in the non-smooth setting.

Sormani in~\cite{Sormani-fund-group} proves that a Riemannian manifold $M^n$ ($n \ge 3$) with $Ric \ge 0$ has a finitely generated fundamental group if it has small linear diameter growth,
\be
	\limsup \frac{\diam \partial \left( B (p , r) \right)}{r} <4 S_n,
\ee
where, the universal constant $S_n$ (coming from Abresch-Gromoll's excess estimate) is 
\be
	S_n := \left(4\cdot 3^n\left(\frac{n-1}{n-2}\right)^{n-1}\frac{n-1}{n}\right)^{-1}.
\ee

This result was later extended to smooth metric measure spaces with non-negative Bakery-Emery Ricci curvature (see Wei~\cite{WeiWylie}).

Following the recent progress in the study of metric-measure spaces having curvature bounded from below in the sense of Lott-Sturm-Villani that are also infinitesimally Hilbertian (having linear Laplacian~\cite{Ambrosio-Gigli-Savare}), come many tools that were previously only available in the Riemannian setting. Among these tools, are the splitting theorem and Abresch-Gromoll excess estimates~\cite{Gigli-Mosconi-Excess} and~\cite{Gigli-Splitting}, to name a few.

Our purpose in this paper is to extend the above mentioned result of Sormani~\cite{Sormani-fund-group} to the spaces satisfying the curvature-dimension condition $CD(0,N)$ that are also infinitesimally Hilbertian (in short, $RCD(0,N)$ spaces). In the course of the proof, it will become clear that we need to assume some other metric conditions on the space but the general approach is reminiscent of that of \cite{Sormani-fund-group}. The main theorem of this paper is  the following:
\bt\label{thm-main}
   Let $\left( X , d_X , m \right)$ be a connected, locally contractible, and non-branching geodesic metric-measure space with $\operatorname{supp}(m) = X$. Suppose $X$ satisfies the $CD(0,N)$ curvature-dimension conditions that is also infinitesimally Hilbertian (see~\cite{Ambrosio-Gigli-Savare} for the detailed definition). If $X$ has small linear diameter growth
   \be
   		\limsup \frac{\diam \partial \left( B (p , r) \right)}{r} <4 S_N,
   \ee
where,
\be
 S_N = 
  \begin{cases}
  \left(9\frac{N-1}{2-N}+4\right)^{-1}&\text{ if $1<N<2$},\\
  \frac{1}{13}&\text{ if $N=2$},\\
  \left(4+2\cdot 3^N\left(\frac{N-1}{N-2}\right)^{N-1}\frac{N-1}{N}\right)^{-1}&\text{ if $N>2$}; 
 \end{cases}
\ee 
 Then,
 $X$ has finitely generated fundamental group.  
\et

Since, the fundamental group and the diameter growth in nature are independent from the measure on the space, we can rephrase our main theorem in the following different but more enlightening manner:
\bc\label{cor-main1}
 	Let $\left( X , d_X \right)$ be a connected, locally contactible and non-branching geodesic metric space with small linear diameter growth. If one can find a Borel measure $m$ on $X$ with $\operatorname{supp}(m) = X$ and for which $\left( X , d_X , m \right)$ becomes an $RCD(0,N)$ space, then, $\pi_1(X)$ is finitely generated.
\ec
\br
	One might be interested to use the Corollary~\ref{cor-main1} to produce many non-branching examples (and not necessarily coming from Finsler manifolds) of metric spaces that do not posses any Borel measures with full support that would make the space an $RCD(0,N)$ space.
\er

When the underlying space is non-negatively curved in the Alexandrov sense, our result simplifies to:
\bc\label{cor-main2}
Let $X$ be a metric space with non-negative curvature in the sense of Alexandrov, then, $X$ has finitely generated fundamental group if $X$ has small diameter growth.
\ec

\br
 Perelman~\cite{PerelmanAlexII} proved that any non-compact non-negatively curved Alexandrov space $X$ has a closed 
 totally convex subset $S$, which is a deformation retraction of $X$. 
 Thus, the fundamental group of $X$ is isomorphic to the one of $S$, accordingly it is finitely generated. 
 Corollary \ref{cor-main2} gives a different proof for Alexandrov spaces with small diameter growth. 
\er

This paper is organized as follows: Section \ref{sec-RCD-Excess} is devoted to a brief review of excess estimates in non-smooth setting; In Section \ref{sec-Univ-Cover}, we will discuss universal coverings of $RCD(K,N)$ spaces and their properties; In Section \ref{sec-halfway-cut-lemmas}, we generalize the half way and uniform cut lemmas to non-smooth spaces and the proofs of Theorem \ref{thm-main} and Corollary \ref{cor-main2} are presented in Section \ref{sec-proof}.






\section*{Acknowledgements}

The authors would like to thank professor Karl Theodor Sturm and the stochastic analysis group at the University of Bonn.

Yu Kitabeppu is partly supported by the Grant-in-Aid for JSPS Fellows, The Ministry of Education, Culture, Sports, Science and Technology, Japan, and also partly supported by the grants of Strategic Young Researcher Overseas Visits Program for Accelerating Brain Circulation during his stay in Bonn. Sajjad Lakzian is supported by the postdoctoral fellowship at the Hausdorff Institute for Mathematics, University of Bonn. Sajjad Lakzian would also like to thank professor Christina Sormani for teaching the authors about this result.





\section{$RCD(0,N)$ Spaces and Excess Estimates}\label{sec-RCD-Excess}




\subsection{Abresch-Gromoll Excess Estimates}
Let $M$ be a complete Riemannian manifold. 
Take two distinct points $y_1 , y_2 \in M$ and fix them,  then for any $x \in M$, the excess , $e(x)$ is
\be
	e(x) := d(x , y_1) + d(x , y_2) - d(y_1 , y_2).
\ee
It is straightforward that $e(x)$ is a Lipschitz function with Lipschitz constant $2$.

Now suppose $\gamma$ is a minimal geodesic connecting $y_1$ and $y_2$ and define the leg and height functions $l(x)$ and $h(x)$ (resp.) as 
\be
	l(x) := \min \left\{ d(x,y_1) , d(x , y_2) \right\} \;\;\; and \;\;\; h(x) := \min_t d(x , \gamma(t)).
\ee
The triangle inequality implies $e(x) \le 2 h(x)$. 

The significance of Abresch-Gromll excess estimate is that they give a non-trivial upper bound for the excess that has the right asymptotic behavior. 

Abresch-Gromoll~\cite{AbreschGromoll} prove that when $Ric \ge 0$ and when $h(x) \le \frac{l(x)}{2}$, then
\be
	e(x) \le 4 \left( \frac{h^n(x)}{l(x)} \right)^{n-1}.
\ee




\subsection{RCD(0,N) Spaces}\label{subsec-RCD}

Sturm in~\cite{Sturm-06} and~\cite{Sturm-2006-II} and Lott-Villani~\cite{Lott-Villani-09} independently developed a notion of a metric measure space having Ricci curvature being bounded from below by $K \in \R$ and dimension bounded above by $N \in [0 , \infty]$. The conditions that these spaces must satisfy are called the curvature-dimension conditions and these spaces are called to be of class $CD(K,N)$.

These curvature-dimension bounds actually generalize the smooth Ricci curvature bounds for Riemannian manifolds. Another nice property of the $CD(K,N)$ classes is their closedness under measured Gromov-Hausdorff convergence (c.f.~\cite{Lott-Villani-09}).

In order to get a "local-to-global" property on top of the aforementioned properties, Bacher-Sturm in~\cite{BacherSturm} introduce a variation of the curvature-dimension conditions which is called the reduced curvature-dimension condition i.e. $CD^*(K,N)$. 

The $CD^*(K,N)$ condition while being a local condition can be realized by some non-linear Finsler structures. It is well-known (see Cheeger-Colding~\cite{ChCo-PartI}~\cite{ChCo-PartII}) that non-linear Finsler structures do not arise as limits of Riemannian structures with Ricci curvature bounded below and they can exhibit undesirable behaviors that do not match with our expectations of a space with curvature bounded below.

To exclude these non-linear anomalies, Ambrosio-Gigli-Savare in \cite{Ambrosio-Gigli-Savare} define the notion of a space being "infinitesimally Hilbertian" spaces. To wit, "infinitesimally Hilbertian" means that the space enjoys a linear Laplacian or equivalently the Sobolev space $W^{1,2} (X , d_X , m)$ is Hilbert. An $RCD(K,N)$ space is a $CD(K,N)$ space which is also infinitesimally Hilbertian. 

$RCD(K,N)$ condition is again stable under measured Gromov-Hausdorff convergence and is also compatible with the smooth Riemannian setting. Infinitesimally Hilbertian spaces also benefit from a very key property namely
\be
	Infinitesimally\; Hilbertian + CD^*(0,N) = Infinitesimally \;Hilbertian + CD(0,N).
\ee
See~\cite{BacherSturm} for a proof.




\subsection{Excess Estimates for RCD(K,N) Spaces}\label{subsec-excess-nonsmooth}

Gigli-Mosoni in~\cite{Gigli-Mosconi-Excess} proves Abresch-Gromoll type excess estimates for $RCD(0,N)$ spaces. They also generalize Cheeger-Colding's excess estimates that appeared in~\cite{ChCo-almost-rigidity}. For the sake of clarity, we will outline Gigli-Mosconi's result in below.

Let $\left(X , d_X , m  \right)$ be an $RCD(K,N)$ space for some $K \le 0$ and for $1 < N < \infty$. Let $\bar{x} \in \operatorname{supp}(m)$ be a point in the support of the background measure. Furthermore assume that the leg and height functions satisfy
\be
	h(\bar{x}) < l(\bar{x}),
\ee
then,
\be\label{eq-excess}
e(\bar x)\leq 
\begin{cases}
2\frac{N-1}{N-2}\left(D_{K, N}(\bar x)h^{N}(\bar x)\right)^{\frac{1}{N-1}}&\text{if $N>2$},\\
\frac{N}{2-N}D_{K, N}(\bar x)h^{2}(\bar x)&\text{if $1<N<2$},\\
D_{K, N}(\bar x)h^{2}(\bar x)\left(\frac{1}{1+\sqrt{1+D^{2}(\bar x)h^{2}(\bar x)}}+\log\frac{1+\sqrt{1+D^{2}(\bar x)h^{2}(\bar x)}}{D_{K, N}(\bar x) h(\bar x)}\right)&\text{if $N=2$},
\end{cases}
\ee
where, 
\be
D_{K, N}(\bar x)=\left(\frac{s_{K, N}(h(\bar x))}{h(\bar x)}\right)^{N-1}\frac{c_{K, N}(l(\bar x)-h(\bar x))}{N};
\ee 
\be
s_{K,N}(\theta)=
\begin{cases}
\sqrt{\frac{N-1}{K}}\sin\left(\theta\sqrt{\frac{K}{N-1}}\right)&\text{ if $K>0$},\\
\theta&\text{ if $K=0$},\\
\sqrt{\frac{N-1}{-K}}\sinh\left(\theta\sqrt{\frac{-K}{N-1}}\right)&\text{ if $K<0$};
\end{cases}
\ee
and
\be
c_{K, N}(\theta)=
\begin{cases}
\dfrac{N-1}{\theta}, &\text{if $K=0$},\\
&\\
\sqrt{-K(N-1)}{\rm cotanh}\left(\theta\sqrt{\frac{-K}{N-1}}\right), &\text{if $K<0$}.
\end{cases}
\ee
When $K=0$, these estimates simplify to 
\be
 e(\bar x)\leq 
 \begin{cases}
\frac{N-1}{2-N}\frac{h^2(\bar x)}{l(\bar x)-h(\bar x)}&\text{ if $1<N<2$},\\
\frac{1}{2}\frac{h^2(\bar x)}{l(\bar x)-h(\bar x)}
\left(\frac{1}{1+\sqrt{1+\left(\frac{1}{2}\frac{h^2(\bar x)}{l(\bar x)-h(\bar x)}\right)^2}}+\mathrm{log}
\frac{1+\sqrt{1+\left(\frac{1}{2}\frac{h^2(\bar x)}{l(\bar x)-h(\bar x)}\right)^2}}{\frac{1}{2}\frac{h^2(\bar x)}{l(\bar x)-h(\bar x)}}\right)&\text{ if $N=2$},\\
2\frac{N-1}{N-2}\left(\frac{N-1}{N}\frac{h^N(\bar x)}{l(\bar x)-h(\bar x)}\right)^{\frac{1}{N-1}}&\text{ if $N>2$}.
\end{cases}
\ee






\section{Universal Covers of $RCD(0,N)$ Spaces}\label{sec-Univ-Cover}

In this section, we will discuss the properties and natural metric measure structure of the universal cover of an $RCD(0,N)$ space. 

Let $X$ be a topological space then, a covering $P : \tilde{X} \rightarrow X$ is called the universal cover if $\tilde{X}$ is simply connected. It is well known that any other covering of $X$ can itself be covered by the universal cover. 

For existence of the universal cover we only need to require very mild topological conditions . In fact, if $X$ is connected, locally pathwise connected and 	semi-locally simply connected, then a universal cover of $X$ exists (see~\cite{BacherSturm} for details).

In this paper we will need to be able to apply the excess estimates~(\ref{eq-excess}) (see also ~\cite{Gigli-Mosconi-Excess}) to a universal covering of an $RCD(0,N)$ metric measure space $X$. Hence, we will need a canonical metric measure structure on a universal covering of a metric measure space $X$.

Let $\left( X , d_X , m \right)$ be a metric measure space and let $P: \tilde{X} \rightarrow X$ be a universal covering.




\subsection*{Canonical Metric, $\tilde{d}_{\tilde{X}}$, on $\tilde{X}$} A curve $\tilde{\gamma}$ in $\tilde{X}$ is called admissible whenever $ \gamma := P \circ \tilde{\gamma}$ is a continuous curve in $X$. For a pair of points $\tilde{x} , \tilde{y} \in \tilde{X}$, the metric $\tilde{d}_{\tilde{X}}$ is defined as
\be
	\tilde{d}_{\tilde{X}} \left( \tilde{x} , \tilde{y} \right) := \inf \left\{ Length\left( \tilde{\gamma} \right) \; | \; \tilde{\gamma} \; \text{is admissible and connects} \; \tilde{x} \; \text{to} \; \tilde{y} \right\}.
\ee

Notice that $Length\left( \tilde{\gamma} \right)$ is computed using the length structure of the base space , $X$ and the fact that $\tilde{X}$ is locally homeomorphic to $X$.
The covering map $P: (\tilde{X},\tilde{d}_{\tilde{X}})\rightarrow (X,d)$ 
becomes a local isometry and 1-Lipschitz map.




\subsection*{Canonical Measure, $\tilde{m}$, on $\tilde{X}$}
Again using the properties of a covering map, one can canonically obtain a measure, $\tilde{m}$, on the covering space, $\tilde{X}$. Let $\tilde{A} \subset \tilde{X}$ be any subset such that the restriction of the covering map $P$ to $\tilde{A}$ is an isometry to $P(\tilde{A})$. Define $\tilde{m}(\tilde{A}) := m(P(\tilde{A}))$ and then extend this measure to the $\sigma-$algebra generated by all such sets, which in turn is equal to the Borel $\sigma-$algebra of $\tilde{X}$ (for details see~\cite{BacherSturm}).  

The measure $\tilde{m}$ can also be defined in the following equivalent manner
\be
	\tilde{m} (\tilde{U}) := \sup \left\{\sum m (P(\tilde{A}_j))  \; | \; \tilde{U} = \sqcup \tilde{A}_j \right\}.
\ee

\begin{thm} 
$\left(\tilde{X} , \tilde{d}_{\tilde{X}} , \tilde{m}   \right)$ is an $RCD(0,N)$ space whenever $\left(  X , d_X , m \right)$ is an $RCD(0,N)$ space.
\end{thm}

\begin{proof}
 Since $(X,d,m)$ is an $RCD(0,N)$ space, it is also an $RCD^*(0,N)$ space, namely it is infinitesimally Hilbertian and a $CD^*(0,N)$ space. 
 Both properties are the local property. Hence, by the construction of $\tilde{d}_{\tilde{X}}$ and $\tilde{m}$, $(\tilde{X},\tilde{d}_{\tilde{X}},\tilde{m})$ is also an $RCD^*(0,N)$ space. 
 However the conditions $CD^*(0,N)$ and $CD(0,N)$ are equivalent to each other. Then $(\tilde{X},\tilde{d}_{\tilde{X}},\tilde{m})$ is an $RCD(0,N)$ space. 
\end{proof}






\section{Half Way Lemma and Uniform Cut Lemma in Non-Smooth Setting}\label{sec-halfway-cut-lemmas}




\subsection{Half Way Lemma}
To apply the Half way Lemma in our setting, we need one more assumption on a metric space $X$, 
namely locally contractibility, which guarantees the locally semi-simply connectedness of $X$. 
Accordingly the existence of the universal cover is guaranteed. 
\bl[Halfway Lemma]\label{lem-Halfway}
 Let $(X,d)$ be a connected and geodesic metric space. 
 Assume furthermore $X$ is locally contractible and proper. Then there exist an ordered set of independent generators $\{ g_1 , g_2 , \dots  \}$ with minimal representative geodesic loops $\gamma_k$ with $ Length \left( \gamma_k \right) = d_k$ such that
 \be\label{halfwayestimate}
 	d_X \left( \gamma_k(0) , \gamma_k \left( \frac{d_k}{2} \right)  \right) = \frac{d_k}{2},
 \ee
 and if $\pi_1(X , x_0)$ is infinitely generated, one obtains a sequence of such generators.
\el
\begin{proof}
 First we note that $\tilde{X}$ is proper if and only if so is $X$. Now fix $x_0\in X$ and let $\tilde{x}_0\in\tilde{X}$ be a lift of $x_0$ to $\tilde{X}$. Obviously, for any non-trivial element $g\in G=\pi_1(X,x_0)$ one has $\tilde{d}_{\tilde{X}}(\tilde{x}_0,g\tilde{x}_0)>0$. Furthermore, the locally contractibility and the properness of $X$ guarantee the positivity of a minimal value of $\tilde{d}_{\tilde{X}}(\tilde{x}_0,g\tilde{x}_0)$.

Since $G$ is discrete, there exists an element $g_1\in G$ attaining the minimum. Now we can proceed by induction just as in~\cite{Sormani-fund-group}. 
And we also obtain (\ref{halfwayestimate}) as in~\cite{Sormani-fund-group}.
\end{proof}
 
\br
 An $RCD(0,N)$ space $X$ is proper. However it is not known to authors whether an $RCD(0,N)$ is also automatically locally contractible or not.
\er




\subsection{Uniform Cut Lemma}
To generalize the uniform cut lemma of~\cite{Sormani-fund-group} to our setting, we need some non-branching assumptions on $RCD(0,N)$ space $X$. Moreover we need modify the value of $S_N$. 

Define the universal constant $S_N$ by 
\be\label{eq-universal-constant}
 S_N = 
 \begin{cases}
  \left(9\frac{N-1}{2-N}+4\right)^{-1}&\text{ if $1<N<2$},\\
  \frac{1}{13}&\text{ if $N=2$},\\
  \left(4+2\cdot 3^N\left(\frac{N-1}{N-2}\right)^{N-1}\frac{N-1}{N}\right)^{-1}&\text{ if $N>2$}. 
 \end{cases}
\ee 

We define the set of cut locus for $x \in X$, roughly-speaking, as the set of points at which geodesics emanating from $x$ stop being minimizing.
\bd[Cut Locus] 
 Let $(X,d)$ be a complete geodesic metric space. 
 For a given point $x\in X$, we define the set of cut locus at $x$ by 
 \be
  C_x:=\left\{ y\in X\, ;\, ^{\nexists}z\in X \,\mathrm{s.t.}\, d(x,z)=d(x,y)+d(y,z)\right\}.
 \ee
 An element in $C_x$ is called a cut point.
 For Riemannian manifolds, they coincide with the ordinal cut locus.
\ed
\br
Since our space is a geodesic space, it is straightforward to see that the notion of cut locus as defined here coincides with the minimal cut locus as defined in~\cite{ShkSor1}. The interested reader should consult~\cite{ShkSor1} for a thorough discussion about conjugate and cut points in length spaces.
\er
\bl[Uniform Cut Lemma for $N \neq 2$]\label{lem-cut}
 Let $(X,d,m)$ be a non-branching $RCD(0,N)$ space with $N\neq 2$. Let $\gamma$ be a non-contractible geodesic loop based at a point $x_0 \in X$ with $Length (\gamma) = L$. Suppose the following are true
 
 \textbf{(a)} $\gamma$ has the shortest length among all loops homotopic to $\gamma$.
 
 \textbf{(b)} $\gamma$ is minimal on both intervals $\left[ 0 , \frac{L}{2} \right]$ and $\left[ \frac{L}{2} , L \right]$.
 
 Then, for any $x \in \partial B(x_0 , RL)$ with $L \ge \frac{1}{2} + S_N$, one has
\be
	d_X \left( x , \gamma \left( \frac{L}{2} \right)  \right) \ge \left(  R - \frac{1}{2} \right)L  + 2 S_N L.
\ee 
 where, $S_N$ is the universal constant defined by (\ref{eq-universal-constant}).
\el
For the sake of completeness, we outline the proof (which is similar to the one in~\cite{Sormani-fund-group}) in below.
\begin{proof}
 Throughout the proof we have $N \neq 2$. We first observe that for a geodesic loop $\gamma : [0,L] \rightarrow X$ based at a point $x_0$, There does not exist geodesic from $x_0$ through $\gamma \left( \frac{L}{2} \right)$ such that it is still minimal after passing through $\gamma\left( \frac{L}{2} \right)$. On the contrary, suppose that there exists such geodesic $\eta : [0,L/2+\epsilon]\rightarrow X$. Both curves $\eta \left( \frac{L}{2}\rightarrow \frac{L}{2}+\epsilon \right) \circ \gamma \left( 0\rightarrow \frac{L}{2}  \right)$ and $\eta \left( \frac{L}{2}\rightarrow \frac{L}{2}+\epsilon \right) \circ \gamma \left( L \rightarrow \frac{L}{2} \right)$ are minimal geodesics and this contradicts the non-branching property of $X$. 
The above claim means that $\gamma \left( \frac{L}{2} \right)\in C_{x_0}$ and $d(x,\gamma(L/2))>L/2+RL$ for any $x\in\partial B(x_0,RL)$, $R>1/2$.  
 
For $R_0 = \frac{1}{2} + L$, we will examine the proof of uniform cut lemma in our setting. 
Suppose there exists a point $x \in \partial B(x_0 , R_0L)$ with 
\be\label{eq-h-estimate}
d_X \left(  x , \gamma \left( \frac{L}{2} \right) \right) =: A < 3S_N L.
\ee
Let $\beta: [0 , A] \rightarrow X$ be a minimal geodesic from $\gamma \left( \frac{L}{2} \right)$ to $x$.
  Consider the triangle in $\tilde{X}$ with vertices $\tilde{x_0} , g \tilde{x_0}$ and $\tilde{x}$ and with geodesic legs given by the lifts $\tilde{\gamma}$ from $\tilde{x_0}$ to $g \tilde{x_0}$ and $\tilde{\beta}(0 \to A) \circ \tilde{\gamma} \left( 0 \to \frac{L}{2} \right)$ from $\tilde{x_0}$ to $\tilde{x}$ . Let 
 \be
\tilde{ l_0} := \tilde{d}_{\tilde{X}} \left( \tilde{x} , \tilde{x_0} \right) \ge d_X (x , x_0) = R_0 L,
\ee 
 and
\be 
	\tilde{l_1 }:= \tilde{d}_{\tilde{X}} \left( \tilde{x}  , g\tilde{x_0}\right) \ge d_X (x , x_0) = R_0 L.
\ee 
 
 Now on one hand, the excess at $\tilde{x}$ satisfies
\be
	e(\tilde{x}) := \tilde{ l_0} + \tilde{l_1} - \tilde{d}_{\tilde{X}} \left( \tilde{x_0} , g \tilde{x_0} \right) \ge 2R_0L - L = 2S_N L,
\ee
 so we can apply the non-smooth excess estimates. And on the other hand, since $S_N < \frac{1}{8}$ one observes that
 \be
 	l(\tilde{x}) - h(\tilde{x}) \ge \left(\frac{1}{2}+S_N\right)L-3S_NL=L\left(\frac{1}{2}-2S_N\right)> \frac{L}{4}. 
 \ee

Now, applying the Abresch-Gromoll type excess estimates for $RCD(0,N)$ spaces (see \cite{Gigli-Mosconi-Excess}) yields\\ 
 
\be
    2S_NL\le e(\tilde{x}) < 
\begin{cases}
 \frac{N-1}{2-N}\frac{(3S_NL)^2}{L\left(\frac{1}{2}-2S_N\right)}&\text{ if $1<N<2$},\\
 2\frac{N-1}{N-2}\left(\frac{N-1}{N}\frac{(3S_NL)^N}{L\left(\frac{1}{2}-2S_N\right)}\right)^{\frac{1}{N-1}}&\text{ if $N>2$}.
\end{cases}
.
\ee
\\

The above inequalities simplify to 
\be
\begin{cases}
 S_N>\left(9\frac{N-1}{2-N}+4\right)^{-1}&\text{ if $1<N<2$},\\
 S_N>\left(2\cdot3^N\frac{N-1}{N}\left(\frac{N-1}{N-2}\right)^{N-1}+4\right)^{-1}&\text{ if $N>2$}.
\end{cases}
\ee

Both inequalities contradict the definition of $S_N$.

 For $R \ge R_0$ , take $y \in \partial B (x_0 , R_0L) \cap \gamma \left( 0 \rightarrow \frac{L}{2} \right)$ and compute
\begin{eqnarray}
d_X \left( x , \gamma \left( \frac{L}{2} \right) \right) &=& d_X \left( x , y \right) + d_X \left( y ,  \gamma \left(\frac{L}{2} \right)\right) \\ & \ge&  \left( RL - R_0L  \right) + 3S_N L = \left(R - \frac{1}{2}  \right)L + 2S_N L.
\end{eqnarray}

\end{proof}

\bl[Uniform Cut Lemma for $N=2$]\label{lem-cut-N2}
  Let $(X,d,m)$ be a non-branching $RCD(0,2)$ space. Let $\gamma$ be a non-contractible geodesic loop based at a point $x_0 \in X$ with $Length (\gamma) = L\ge 1352$. 
  Suppose the following are true
 
 \textbf{(a)} $\gamma$ has the shortest length among all loops homotopic to $\gamma$.
 
 \textbf{(b)} $\gamma$ is minimal on both intervals $\left[ 0 , \frac{L}{2} \right]$ and $\left[ \frac{L}{2} , L \right]$.
 
 Then, for any $x \in \partial B(x_0 , RL)$ with $R \ge \frac{1}{2} + S_2$, one has
\be
	d_X \left( x , \gamma \left( \frac{L}{2} \right)  \right) \ge \left(  R - \frac{1}{2} \right)L  + 2 S_2 L,
\ee 
where, $S_2$ is the universal constant defined in (\ref{eq-universal-constant}).
\el
\begin{proof}
 In the same way as in the proof of Lemma \ref{lem-cut}, we have 
\begin{align}
 2S_2L
&\le e(\tilde x)\notag\\
&\le \frac{1}{2}\frac{h^2(\tilde x)}{l(\tilde x)-h(\tilde x)}
\left(\frac{1}{1+\sqrt{1+\left(\frac{1}{2}\frac{h^2(\tilde x)}{l(\tilde x)-h(\tilde x)}\right)^2}}+\mathrm{log}
\frac{1+\sqrt{1+\left(\frac{1}{2}\frac{h^2(\tilde x)}{l(\tilde x)-h(\tilde x)}\right)^2}}{\frac{1}{2}\frac{h^2(\tilde x)}{l(\tilde x)-h(\tilde x)}}\right)\notag\\
&\le \frac{1}{2}\frac{h^2(\tilde x)}{l(\tilde x)-h(\tilde x)}\left(\frac{1}{2}+\frac{2+\frac{1}{2}\frac{h^2(\tilde x)}{l(\tilde x)-h(\tilde x)}}{\frac{1}{2}\frac{h^2(\tilde x)}{l(\tilde x)-h(\tilde x)}}\right)\label{ex-estimate2}\\
&=\frac{1}{2}\frac{h^2(\tilde x)}{l(\tilde x)-h(\tilde x)}\left(\frac{3}{2}+4\frac{l(\tilde x)-h(\tilde x)}{h^2(\tilde x)}\right)\label{ex-estimate2}. \notag
\end{align}
Let $t_0\in [0,1]$ be a point satisfying $h(\tilde x)=\tilde{d}(\tilde{x},\tilde{\gamma}(t_0L))$. 
Without loss of generality, we may assume that $t_0\in [0,1/2]$. Then one can bound $h(\tilde x)$ from below as
\be
h(\tilde x)\ge d(x,\gamma(t_0L))\ge d(x,\gamma(0))-d(\gamma(0),\gamma(t_0L))\ge S_2L;
\ee
On the other hand, we have 
\be
l(\tilde{x})-h(\tilde{x})\le \tilde{d}(\tilde{x},g\tilde{x_0})-\tilde{d}(\tilde{x},\tilde{\gamma}(t_0L))
\le \tilde{d}(\tilde{\gamma}(L),\tilde{\gamma}(t_0L))=(1-t_0)L,
\ee
therefore,
\be
 \frac{l(\tilde{x})-h(\tilde{x})}{h^2(\tilde{x})}\le \frac{(1-t_0)L}{S_2^2L^2}\le \frac{1}{S_2^2L}\le \frac{1}{\frac{1}{169}\cdot 1352}=\frac{1}{8}.\label{appestimate2}
\ee
Combining the two inequalities (\ref{ex-estimate2}) and (\ref{appestimate2}), one obtains
\begin{align}
 2S_2L
&\le \frac{1}{2}\frac{h^2(\tilde x)}{l(\tilde x)-h(\tilde x)}\left(\frac{3}{2}+4\frac{l(\tilde x)-h(\tilde x)}{h^2(\tilde x)}\right)\notag\\
&<\frac{1}{2}\frac{(3S_2L)^2}{L\left(\frac{1}{2}-2S_2\right)}\left(\frac{3}{2}+4\cdot\frac{1}{8}\right)\notag\\
&=2\frac{9S_2^2L}{1-4S_2},
\end{align}
or
\be
 1<\frac{9S_2}{1-4S_2}.
\ee
This is a contradiction to the definition of $S_2$. 
\end{proof}

\br
As we see in Lemma \ref{lem-cut-N2}, when $N = 2$, the uniform cut lemma holds only for loops whose lengths are sufficiently large ($ L > 1352$), but this is enough for us since in the proof of Theorem \ref{thm-main}, we have a sequence of loops, the lengths of which are diverging to $\infty$.
\er

\br
	It is known that an $RCD(0,N)$ spaces are strongly $CD(0,N)$ which implies that they are essentially non-branching (see ~\cite{Rajala-Sturm} for details) but this is not strong enough to get a topological result as in these notes.
\er






\section{Small Diameter Theorem}\label{sec-proof}
In this section we prove Theorem \ref{thm-main} and Corollary \ref{cor-main2}

\subsection*{Proof of Theorem \ref{thm-main}}

Now that we have all the essential ingredients (Half Way and Uniform Cut Lemmas), the proof the main theorem essentially goes verbatim as in the proof of the small diameter growth theorem in~\cite{Sormani-fund-group}. For the sake of completeness, we will repeat the proof in below:

Suppose, $\pi_1(X , x_0)$ is infinitely generated. Construct the ordered set of independent generators $g_1 , g_2 , \dots $ as in Lemma \ref{lem-Halfway} with minimal representative loops $\gamma_1 , \gamma_2 , \dots$ (resp.) 

First observation is that $d_k := Length (\gamma_k)$ diverges to infinity since otherwise we would have , for some large $R$,  $\pi_1(X , x_0) = \pi_1 \left( \overline{B(x_0 , R)}  , x_0\right)$ which is finite (since $ \overline{B(x_0 , R)} $ is compact) which is a contradiction.

Let $\{x_k\}$ be a sequence with $x_k \in \partial B(x_0 , \left(\frac{1}{2} + S_N  \right)d_k)$ and let $\beta_k : I \rightarrow X$ be the minimal geodesic from $x_k$ to $x_0$. From the uniform cut lemma (Lemma \ref{lem-cut}) we have 
\be
	d_X \left( x_k , \gamma_k \left( \frac{d_k}{2} \right)  \right) \ge 3S_N d_k.
\ee

Now take the points $y_k \in \partial B \left( x_0 , \frac{d_k}{2} \right) \cap \beta_k(I)$; then, by the triangle inequality, we get
\be
	d_X \left( y_k , \gamma_k \left( \frac{d_k}{2} \right)   \right) \ge 2S_N d_k
\ee
hence,
\be
	\limsup \frac{\diam \partial \left( B (p , r) \right)}{r} \ge \limsup \frac{ d_X \left( y_k , \gamma_k \left( \frac{d_k}{2} \right)   \right)}{\frac{d_k}{2}} \ge 4 S_N
\ee
which is a contradiction.

\textbf{\textit{QED.}}

\subsection*{Proof of Corollary \ref{cor-main2}}
Let $X$ be a metric space with non-negative curvature in the sense of Alexandrov and with small linear diameter growth. 
It is well-known that $X$ is non-branching (for example see~\cite{BBI}). From Petrunin~\cite{PetruninCD-KN}, we know that $X$ is a $CD(0,N)$ space and infinitesimal Hilbertianity follows from Kuwae-Machigashira-Shioya~\cite{KMS-Alex-RCD}. 
The local (Lipschitz) contractibility of $X$ follows from Mitsuishi-Yamagichi~\cite{Mitsuishi-Yamagichi-SLLC} and Perelman~\cite{PerelmanAlexII}. So, $X$ satisfies all the hypotheses of Theorem \ref{thm-main}.

\textbf{\textit{QED.}}






\bibliographystyle{plain}
\bibliography{reference2014}

\end{document}